%% file: BEMcontact.tex
\documentclass[onefignum,onetabnum]{siamart190516}

\usepackage[english]{babel}
\usepackage{amsmath}
\usepackage{amssymb,amsbsy}
\usepackage{amscd}
\usepackage{amsfonts}
\usepackage{graphicx}
\usepackage{epstopdf}
\usepackage{algorithmic}
\ifpdf
  \DeclareGraphicsExtensions{.eps,.pdf,.png,.jpg}
\else
  \DeclareGraphicsExtensions{.eps}
\fi

\newsiamremark{remark}{Remark}
\newsiamremark{hypothesis}{Hypothesis}
\newsiamthm{claim}{Claim}

\usepackage{latexsym}
\usepackage{float}
\usepackage{setspace}
\usepackage{fancyhdr}
\usepackage{color}

\usepackage{bm}
\usepackage{textcomp}
\usepackage{upgreek}
\usepackage{textgreek}
\renewcommand{\uppi}{\text{\textpi}}
\usepackage{amsopn}

\usepackage{tikz}
\usetikzlibrary{arrows}
\usepackage[utf8]{inputenc}
\usepackage{pgfplots}
\usepackage{pbox}

\pgfplotsset{every tick label/.append style={font=\footnotesize}}
\pgfplotsset{label style={font=\footnotesize}}

\usepackage{hyperref}
\usepackage{cleveref}

\usepackage{subfigure}

\crefname{equation}{}{}
\crefname{figure}{Figure}{Figures}
\crefname{section}{Section}{Sections}
\crefname{hypothesis}{Hypothesis}{Hypotheses}

\newcommand{\B}[1]{\bm{#1}}
\newcommand{\x}{{\B{x}}}
\newcommand{\y}{{\B{y}}}
\newcommand{\dx}[1][x]{\,\mathrm{d}#1}

\newcommand{\trace}{\gamma}
\newcommand{\LL}{L}
\newcommand{\Ltwo}{\LL^2}
\newcommand{\pop}[1]{\mathcal{#1}}
\newcommand{\bop}[1]{\mathsf{#1}}
\newcommand{\popV}{\pop{V}}
\newcommand{\popK}{\pop{K}}
\newcommand{\bopV}{\bop{V}}
\newcommand{\bopK}{\bop{K}}
\newcommand{\bopKadj}{{\bop{K}'}}
\newcommand{\bopW}{\bop{W}}
\newcommand{\bopI}{\bop{Id}}

\newcommand\D{_\textup{D}}
\newcommand\N{_\textup{N}}
\newcommand\C{_\textup{C}}

\newcommand{\norm}[1]{\| #1 \|}
\newcommand{\starnorm}[2][]{\norm{#2}_{*#1}}
\newcommand{\Vnorm}[2][]{\norm{#2}_{\productspace{V}#1}}
\newcommand{\Bnorm}[2][]{\norm{#2}_{\form{B}#1}}
\newcommand{\Ltwonorm}[2][\Gamma]{\norm{#2}_{#1}}
\newcommand{\Hnorm}[3][\Gamma]{\norm{#3}_{H^{#2}(#1)}}
\newcommand{\Hstarnorm}[3][\Gamma]{\seminorm{#3}_{H_*^{#2}(#1)}}
\newcommand{\seminorm}[1]{| #1 |}
\newcommand{\Hseminorm}[3][\Gamma]{\seminorm{#3}_{H^{#2}(#1)}}

\newcommand{\average}[1]{\left\{#1\right\}_\Gamma}
\newcommand{\RR}{\mathbb{R}}
\newcommand{\form}[1]{\mathcal{#1}}
\newcommand{\productspace}[1]{\mathbb{#1}}

\newcommand{\betaparam}{\beta}
\newcommand{\Th}{\mathcal{T}_h}

\renewcommand{\emptyset}{\varnothing}

\newcommand{\generalpart}[2]{\left[#2\right]_#1}
\newcommand{\pmpart}[1]{\generalpart{\pm}{#1}}
\newcommand{\negpart}[1]{\generalpart{-}{#1}}
\newcommand{\pospart}[1]{\generalpart{+}{#1}}
\newcommand{\param}{\tau}
\newcommand{\change}{\delta}

\newcommand{\discrete}[1]{%
\ifthenelse{\equal{#1}{u}}{{\tt U}\D}{%
\ifthenelse{\equal{#1}{\lambda}}{{\tt U}\N}{%
\ifthenelse{\equal{#1}{v}}{{\tt V}\D}{%
\ifthenelse{\equal{#1}{\mu}}{{\tt V}\N}{%
{\tt#1} }}}}
}

\newcommand{\changeP}{\change P}
\newcommand{\distance}[3][]{d#1\ifthenelse{\equal{#2}{}}{}{\left(#2,#3\right)}}

\newcommand{\betamin}{\betaparam_{\min}}

\newcommand{\Hspace}[2][\Gamma]{H^{#2}(#1)}
\newcommand{\inner}[3][]{\left\langle#2,#3\right\rangle\ifthenelse{\equal{#1}{}}{}{_{#1}}}
\newcommand{\Ltwoinner}[3][\Gamma]{\inner[#1]{#2}{#3}}

\newcommand{\Poly}{\textup{P}}
\newcommand{\DPoly}{\textup{DP}}
\newcommand{\DjumpPoly}{\widetilde{\textup{DP}}}
\newcommand{\DualPoly}{\textup{DUAL}}

\newcommand\bcpsi{\psi\C}
\newcommand\bcg{g\C}

\usepackage{color}
\definecolor{myorange}{rgb}{0.9568,0.4941,0.1961}
\definecolor{myred}{rgb}{0.9098,0.1294,0.2078}
\definecolor{myblue}{rgb}{0.0352,0.4981,0.6509}
\definecolor{myhyperblue}{rgb}{0.1607,0.3922,0.9}
\definecolor{mygreen}{rgb}{0.2235,0.6353,0.2588}
\definecolor{mygrey}{rgb}{0.3,0.3,0.3}

\definecolor{Ccolor}{HTML}{FFA366}
\definecolor{DCcolor}{HTML}{2EA3D0}
\definecolor{reference}{rgb}{0,0,0}

\newcommand{\Ccolordesc}{orange}
\newcommand{\DCcolordesc}{blue}

\definecolor{red}{rgb}{1,0,0}

\newcommand{\referenceplot}{\addplot [reference, dashed, very thick]}

\newcommand{\Cplot     }[1][2]{\addplot [Ccolor, mark=square*, thick, mark size=#1, mark options={solid,draw=black}]}
\newcommand{\Coneplot  }[1][3]{\addplot [Ccolor, mark=triangle*, thick, mark size=#1, mark options={solid,draw=black}]}
\newcommand{\Ctwoplot  }[1][3]{\addplot [Ccolor, mark=diamond*, thick, mark size=#1, mark options={solid,draw=black}]}
\newcommand{\Cthreeplot}[1][3]{\addplot [Ccolor, mark=pentagon*, thick, mark size=#1, mark options={solid,draw=black}]}

\newcommand{\DCplot     }[1][2]{\addplot [DCcolor, mark=*, thick, mark size=#1, mark options={solid,draw=black}]}
\newcommand{\DConeplot  }[1][3]{\addplot [DCcolor, mark=triangle*, thick, mark size=#1, mark options={solid,draw=black}]}
\newcommand{\DCtwoplot  }[1][3]{\addplot [DCcolor, mark=diamond*, thick, mark size=#1, mark options={solid,draw=black}]}
\newcommand{\DCthreeplot}[1][3]{\addplot [DCcolor, mark=pentagon*, thick, mark size=#1, mark options={solid,draw=black}]}

\newcommand{\onedesc}{triangles}
\newcommand{\twodesc}{diamonds}
\newcommand{\threedesc}{pentagons}

\newcommand{\Cdesc}{{\Ccolordesc} squares}

\newcommand{\DCdesc}{{\DCcolordesc} circles}

\newcounter{prope}

\newenvironment{property}[1][]{
\begin{flushleft}
\stepcounter{prope}
\leftskip1em
\rightskip\leftskip
\textup{\textbf{Property \arabic{prope}} \ifthenelse{\equal{#1}{}}{}{(#1)}:}}
{\end{flushleft}}

\title{.}
\headers{Weak imposition of Signorini boundary conditions on BEM}{E. Burman, S. Frei, M. W. Scroggs}

\title{Weak imposition of Signorini boundary conditions on the boundary element method\thanks{Submitted to the editors 2019-08-15.
\funding{Erik Burman was funded by the EPSRC grant EP/P01576X/1. Stefan Frei was funded by the DFG Research Scholarship 3935/1-1.}}}

\author{Erik Burman\thanks{Department of Mathematics, University College London, UK (\email{e.burman@ucl.ac.uk}).}
\and Stefan Frei\thanks{Department of Mathematics and Statistics, University of Konstanz, Germany (\email{stefan.frei@uni-konstanz.de}).}
\and Matthew W. Scroggs\thanks{Department of Engineering, University of Cambridge, UK (\email{mws48@cam.ac.uk}, \url{http://www.mscroggs.co.uk}).}}

\usepackage{amsopn}

\ifpdf
\hypersetup{
  pdftitle={Weak imposition of Signorini boundary conditions on the boundary element method},
  pdfauthor={E. Burman, S. Frei, M. W. Scroggs}
}
\fi

\graphicspath{{img/}}

\begin{document}

\maketitle

\begin{abstract}
We derive and analyse a boundary element formulation for boundary conditions involving inequalities.
In particular, we focus on Signorini contact conditions.
The Calder\'on projector is
used for the system matrix and boundary conditions are weakly imposed using a
particular variational boundary operator designed using techniques
from augmented Lagrangian methods. We present a complete numerical \emph{a priori} 
error analysis and present some numerical examples to illustrate the theory.
\end{abstract}



\section{Introduction}
The application of Nitsche techniques to deal with variational inequalities
has received increasing interest recently, starting from a series of works 
by Chouly, Hild and Renard for elasticity problems with contact \cite{CH13b}. Their 
approach goes back to an augmented Lagrangian formulation, that has first been introduced 
by Alart \& Curnier \cite{AC91}. 

In a previous paper \cite{BeBuScro18}, we have shown how Nitsche techniques can be used
to impose Dirichlet, Neumann, mixed Dirichlet--Neumann or Robin conditions weakly within
boundary element methods. By using the Calder\'on projector, we were
able to derive a unified framework that can be used for different boundary conditions.

The purpose of this article is to extend these techniques to boundary conditions involving inequalities,
such as Signorini contact conditions.
In particular, we consider the Laplace equation with mixed Dirichlet and Signorini boundary conditions: Find $u$ such that
\begin{subequations}\label{eq:Laplace}
\begin{align}
-\Delta u &=0       &&\text{in } \Omega,\\
u &= g\D            &&\text{on } \Gamma\D,\label{eq:diri_bc}\\
u \leqslant \bcg\quad\text{and}\quad
\frac{\partial u}{\partial\B\nu} &\leqslant \bcpsi
                    &&\text{on } \Gamma\C,\label{eq:contact_bc1}\\
\left(\frac{\partial u}{\partial\B\nu} - \bcpsi\middle)\middle(u-\bcg\right) &= 0
                    &&\text{on } \Gamma\C\label{eq:contact_bc2}.
\end{align}\end{subequations}
Here $\Omega \subset \RR^3$ denotes a polyhedral domain with outward pointing normal
$\B\nu$ and boundary $\Gamma:= \Gamma\D \cup \Gamma\C$.
We assume for simplicity that the boundary between $\Gamma\D$ and $\Gamma\C$ coincides with edges between the faces of $\Gamma$.
Whenever it is ambiguous, we will write $\B\nu_\x$ for the outward pointing normal at the point $\x$.
We assume that $g=\begin{cases}g\D&\text{in }\Gamma\D\\\bcg&\text{in }\Gamma\C\end{cases}\in L^2(\Gamma)$ and $\bcpsi \in \Hspace[\Gamma\C]{1/2}$.

Observe that when $\Gamma\C=\emptyset$, there exists a unique solution
to \cref{eq:Laplace} by the Lax--Milgram lemma.
In the case that meas($\Gamma\C)>0$,
the theory of Lions and Stampacchia \cite{LS67} for variational inequalities yields
existence and uniqueness of solutions.
We assume that $u \in \Hspace[\Omega]{3/2+\epsilon}$, for some $\epsilon>0$.

Boundary element methods for Signorini problems were first studied by Han~\cite{Han87}.
A variational formulation involving the Calder\'on projector was presented in~\cite{Han91}.
An alternative formulation is based on Steklov-Poincar\'e operators~\cite{Steinbach13, ZhangLi16}.
The numerical approaches to solve such formulations include a penalty formulation~\cite{SchmitzSchneider91}, operator splitting 
techniques~\cite{Spann92, ZhangZhu13} or semi-smooth Newton methods~\cite{Steinbach13, ZhangLi16}.
The latter reference includes besides the usual energy norm estimates
an $L^2(\Gamma)$-error estimate based on a duality argument. 
Maischak \& Stephan~\cite{MaischakStephan04} presented \emph{a posteriori} error estimates and an 
$hp$-adaptive algorithm for the Signorini problem.
\emph{A priori} error estimates for a penalty-based $hp$ algorithm were shown by Chernov, Maischak \& Stephan~\cite{Chernov07}.
Recently, an augmented Lagrangian approach has been presented in combination with a semi-smooth Newton method~\cite{ZhangLi16},
and variational inequalities have been successfully used for time-dependent contact problems~\cite{Gimperlein}.

We will consider an approach where the full Calder\'on projector is
used and the boundary conditions are included by adding properly
scaled penalty terms to the two equations. This results in
formulations similar to the ones obtained for weak imposition of
boundary conditions using Nitsche's method \cite{Nit71}. The proposed
framework is flexible and allows for the design of a range of
different methods depending on the choice of weights and residuals.

An outline of the paper is as follows. In \cref{sec:b_op}, we introduce 
the basic boundary operators that will be needed and review some of their properties.
Then, in \cref{sec:Dirichlet}, we introduce the variational framework and review the results 
from \cite{BeBuScro18} for the pure Dirichlet problem.
In \cref{sec:contact}, we show how the framework can be applied to Signorini boundary conditions
and the mixed problem \cref{eq:Laplace}. The method is analysed in \cref{sec:analysis}.
We conclude by showing some numerical experiments in \cref{sec:numerical}.

\section{Boundary operators}\label{sec:b_op}
We define the Green's function for the Laplace operator in $\RR^3$ by
\begin{equation}
G(\x,\y) =
\frac{1}{4\uppi|\x-\y|}.
\end{equation}
In this paper, we focus on the problem in
$\RR^3$. Similar analysis can be used for problems in $\RR^2$, in which case this definition should be replaced by
$G(\x,\y) =-\log|\x-\y|/2\uppi$.

In the standard fashion (see e.g. \cite[Chapter 6]{Stein07}), we define
the single layer potential operator, $\pop{V}:H^{-1/2}(\Gamma)\to H^1(\Omega)$,
and
the double layer potential, $\pop{K}:H^{1/2}(\Gamma)\to H^1(\Omega)$,
for $v\in H^{1/2}(\Gamma)$, $\mu\in H^{-1/2}(\Gamma)$, and $\x\in\Omega\setminus\Gamma$ by
\begin{align}
(\popV\mu)(\x) &:= \int_{\Gamma} G(\x,\y) \mu(\y)\dx[\y]\label{eq:single},\\
(\popK v)(\x) &:= \int_{\Gamma} \frac{\partial G(\x,\y)}{\partial\B\nu_\y} v(\y)\dx[\y]\label{eq:double}.
\end{align}
We define the space $H^1(\Delta,\Omega):=\{v\in H^1(\Omega):\Delta v\in\Ltwo(\Omega)\}$, and
the Dirichlet and Neumann traces, $\trace\D:H^1(\Omega)\to H^{1/2}(\Gamma)$
and $\trace\N:H^1(\Delta,\Omega)\to H^{-1/2}(\Gamma)$, by
\begin{align}
\trace\D f(\x)&:=\lim_{\Omega\ni\y\to\x\in\Gamma}f(\y),\\
\trace\N f(\x)&:=\lim_{\Omega\ni\y\to\x\in\Gamma}\B\nu_\x\cdot\nabla f(\y).
\end{align}

We recall that if the Dirichlet and Neumann traces of a
harmonic function are known, then the potentials \cref{eq:single} and
\cref{eq:double} may be used to reconstruct the function in $\Omega$ using the
following relation.
\begin{equation}\label{eq:represent}
u = -\popK(\trace\D u) + \popV(\trace\N u).
\end{equation}

\noindent It is also known \cite[Lemma 6.6]{Stein07} that for all $\mu \in H^{-1/2}(\Gamma)$, the function
\begin{equation}\label{eq:sing_rec}
u^{\popV}_\mu := \popV \mu
\end{equation}
satisfies $-\Delta u^{\popV}_\mu = 0$ and
\begin{equation}\label{eq:sing_rec_stab}
\Hnorm[\Omega]{1}{u^{\popV}_\mu} \leqslant c \Hnorm{-1/2}{\mu}.
\end{equation}
Similarly \cite[Lemma 6.10]{Stein07}, the function
\begin{equation}\label{eq:doub_rec}
u^{\popK}_v := \popK v
\end{equation}
satisfies $-\Delta u^{\popK}_v = 0$ for all $v \in H^{1/2}(\Gamma)$ and
\begin{equation}\label{eq:doub_rec_stab}
\Hnorm[\Omega]{1}{u^{\popK}_v} \leqslant c \Hnorm{1/2}{v}.
\end{equation}

We define
$\average{\trace\D f}$ and $\average{\trace\N f}$
to be the averages of the interior and exterior Dirichlet and Neumann traces of $f$.
We define the single layer, double layer, adjoint double layer, and hypersingular boundary integral operators,
$\bopV:H^{-1/2}(\Gamma)\to H^{1/2}(\Gamma)$,
$\bopK:H^{1/2}(\Gamma)\to H^{1/2}(\Gamma)$,
$\bopKadj:H^{-1/2}(\Gamma)\to H^{-1/2}(\Gamma)$, and
$\bopW:H^{1/2}(\Gamma)\to H^{-1/2}(\Gamma)$,
by
\begin{subequations}
\begin{align}
(\bopK v)(\x)&:=\average{\trace\D\popK v}(\x),&
(\bopV \mu)(\x)&:=\average{\trace\D\popV\mu}(\x),\\
(\bopW v)(\x)&:=-\average{\trace\N\popK v}(\x),&
(\bopKadj\mu)(\x)&:=\average{\trace\N\popV\mu}(\x),
\end{align}
\end{subequations}
where
$\x\in\Gamma$,
$v\in H^{1/2}(\Gamma)$ and $\mu\in H^{-1/2}(\Gamma)$ \cite[Chapter 6]{Stein07}.

Next, we define the Calder\'on projector by
\begin{equation}\label{eq:calder}
\bop{C}:=
\begin{pmatrix}
(1-\sigma)\bopI-\bopK & \bopV\\
\bopW & \sigma\bopI+\bopKadj
\end{pmatrix},
\end{equation}
where $\sigma$ is defined for $\x\in\Gamma$ by \cite[Equation 6.11]{Stein07}
\begin{equation}
\sigma(\x) = \lim_{\epsilon\to0}\frac1{4\uppi\epsilon^2}\int_{\y\in\Omega:\seminorm{\y-\x}=\epsilon}\dx[\y].
\end{equation}
Recall that if $u$ is a solution of \cref{eq:Laplace} then it satisfies
\begin{equation}\label{eq:calder_id}
\bop{C}\begin{pmatrix}\trace\D u\\\trace\N u\end{pmatrix}=\begin{pmatrix}\trace\D u\\\trace\N u\end{pmatrix}.
\end{equation}

Taking the product of \cref{eq:calder_id} with two test functions, and using the fact that $\sigma=\frac12$ almost everywhere,
we arrive at the following equations.
\begin{align}
\left\langle \trace\D u,\mu \right\rangle_{\Gamma} &= \left\langle (\tfrac12\bopI - \bopK)
  \trace\D u,\mu\right\rangle_{\Gamma} + \left\langle \bopV\trace\N
  u,\mu\right\rangle_{\Gamma} &&\forall \mu \in H^{-1/2}(\Gamma),\label{eq:Calderon_1}\\
\left\langle \trace\N u,v \right\rangle_{\Gamma} &= \left\langle (\tfrac12\bopI + \bopKadj)
  \trace\N u,v \right\rangle_{\Gamma} + \left\langle \bopW \trace\D
  u,v\right\rangle_{\Gamma} &&\forall v \in H^{1/2}(\Gamma).\label{eq:Calderon_2}
\end{align}

For a more compact notation, we introduce $\lambda=\trace\N u$ and $u=\trace\D u$ and the Calder\'on form
\begin{multline}\label{eq:compact_form}
\form{C}[(u,\lambda),(v,\mu)]:=
\left\langle(\tfrac12\bopI-\bopK)
 u,\mu\right\rangle_{\Gamma} + \left\langle \bopV\lambda,\mu\right\rangle_{\Gamma}\\
+
\left\langle (\tfrac12\bopI + \bopKadj)
  \lambda,v \right\rangle_{\Gamma} + \left\langle \bopW u,v\right\rangle_{\Gamma}.
\end{multline}
We may then rewrite \cref{eq:Calderon_1} and \cref{eq:Calderon_2} as
\begin{equation}\label{eq:realtion}
\form{C}[(u,\lambda),(v,\mu)]=\left\langle u,\mu\right\rangle_\Gamma+\left\langle\lambda,v\right\rangle_\Gamma.
\end{equation}

\noindent We will also frequently use the multitrace form, defined by
\begin{equation}\label{eq:mult_trace}
\form{A}[(u,\lambda),(v,\mu)]:=
-\left\langle\bopK
 u,\mu\right\rangle_{\Gamma} + \left\langle \bopV\lambda,\mu\right\rangle_{\Gamma}
+
\left\langle \bopKadj
  \lambda,v \right\rangle_{\Gamma} + \left\langle \bopW u,v\right\rangle_{\Gamma}.
\end{equation}
Using this, we may rewrite \cref{eq:realtion} as
\begin{equation}\label{eq:skewsym_relation}
\form{A}[(u,\lambda),(v,\mu)]=\tfrac12\left\langle u,\mu\right\rangle_\Gamma+\tfrac12\left\langle\lambda,v\right\rangle_\Gamma.
\end{equation}

\noindent To quantify the two traces we introduce the product space
\[
\productspace{V} := H^{1/2}(\Gamma)\times H^{-1/2}(\Gamma)
\]
and the associated norm
\[
\Vnorm{(v,\mu)}:= \Hnorm{1/2}{v}+\Hnorm{-1/2}{\mu}.
\]

The continuity and coercivity of $\form{A}$ are immediate consequences of the properties of the operators 
$\bopV$, $\bopK$, $\bopKadj$ and $\bopW$:

\begin{lemma}[Continuity \& coercivity]\label{lemma:cont_cald}
There exists $C>0$ such that
\begin{align*}
\left|\form{A}[(w,\eta),(v,\mu)]\right| &\leqslant C \Vnorm{(w,\eta)}\Vnorm{(v,\mu)}&&\forall(w,\eta),(v,\mu)\in\productspace{V}.
\end{align*}
There exists $\alpha>0$ such that
\begin{align*}
\alpha\left(
\Hstarnorm{1/2}{v}^2+\Hnorm{-1/2}{\mu}^2
\right)
&\leqslant
\form{A}[(v,\mu),(v,\mu)]
&&\forall(v,\mu)\in\productspace{V}.
\end{align*}
\end{lemma}
\begin{proof}See \cite{BeBuScro18}.\end{proof}

\section{Discretisation and weak imposition of Dirichlet boundary conditions}
\label{sec:Dirichlet}
In this section, we introduce the discrete spaces and review briefly how (non-homogeneous) Dirichlet
boundary conditions can be imposed weakly within the variational 
formulations introduced above. For a detailed derivation, and for different
boundary conditions, we refer to \cite{BeBuScro18}.

To reduce the number of constants that appear,
we introduce the following notation.
\begin{itemize}
\item If $\exists C>0$ such that $a\leqslant Cb$, then we write $a\lesssim b$.
\item If $a\lesssim b$ and $b\lesssim a$, then we write $a\eqsim b$.
\end{itemize}

We assume that $\Omega$ is a polygonal domain with faces denoted by
$\{\Gamma_i\}_{i=1}^M$.
We introduce a family of conforming, shape regular triangulations of
$\Gamma$, $\{\Th\}_{h>0}$, indexed by the largest element
diameter of the mesh, $h$.
We let $T_1,..,T_m\in\Th$ be the triangles of a triangulation.

We consider the following finite element spaces
\begin{align*}
\Poly^k_h(\Gamma) &:= \{v_h \in C^0(\Gamma): v_h\vert_{T_i} \in \mathbb{P}_k(T_i) \text{, for every }T_i\in\Th\},\\
\DPoly^l_h(\Gamma) &:= \{v_h \in \Ltwo(\Gamma): v_h\vert_{T_i} \in \mathbb{P}_l(T_i) \text{, for every }T_i\in\Th\},\\
\DjumpPoly^l_h(\Gamma) &:= \{v_h \in \DPoly^l_h(\Gamma): v_h\vert_{\Gamma_i} \in
C^0(\Gamma_i)\text{, for }i=1,\hdots,M\},
\end{align*}
where $\mathbb{P}_k(T_i)$ denotes the space of polynomials of order less
than or equal to $k$ on the triangle $T_i$.

In addition, we consider the space $\DualPoly_h^0(\Gamma)$ of piecewise constant functions on the barycentric dual grid, as shown in \cref{fig:barycentric_dual}.
On non-smooth domains, these spaces have lower order approximation properties than the standard space $\DPoly_h^0(\Gamma)$, as given in the following lemma.

\begin{figure}
\centering
\input{img/grid1.tex}%
\input{img/grid2.tex}%
\input{img/grid3.tex}
\caption{A grid (left), the barycentric refinement of the grid (centre), and the dual grid (right).
In a typical example, the initial grid will not be flat, and so the elements of the dual grid will not necessarily be flat.}
\label{fig:barycentric_dual}
\end{figure}
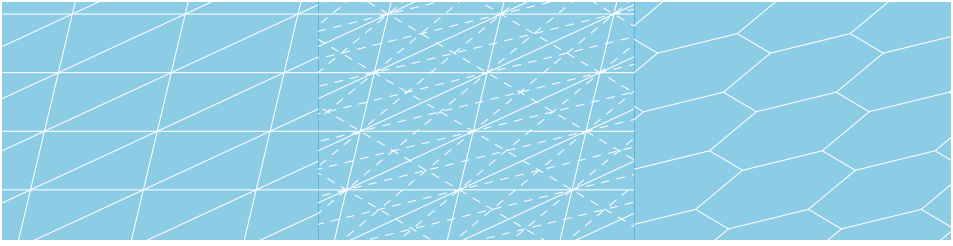

\begin{lemma}\label{lemma:DUAL0app}
Let $\mu\in\Hspace{s}$.
If $\Gamma$ consists of a finite number of smooth faces meeting at edges, then
\begin{align*}
\inf_{\eta_h \in \DualPoly_h^0(\Gamma)}
    \Hnorm{-1/2}{\mu-\eta_h}
&\lesssim h^{\xi+1/2} \Hnorm{\xi}{\mu}
\end{align*}
where $\xi=\min(\tfrac12,s)$.
If $\Gamma$ is smooth, then the same result holds with $\xi=\min(1,s)$.
\end{lemma}
\begin{proof}
See \cite[Appendix 2]{ScroggsThesis}.
\end{proof}

We observe that $\Poly^k_h(\Gamma) \subset H^{1/2}(\Gamma)$, $\DPoly^l_h(\Gamma)
\subset \Ltwo(\Gamma)$, $\DjumpPoly^l_h(\Gamma)
\subset \Ltwo(\Gamma)$, and $\DualPoly_h^0(\Gamma)\subset\Ltwo(\Gamma)$. We define the discrete product
space 
\begin{align*}
   \productspace{V}_h :=  \Poly^k_h(\Gamma) \times \Lambda_h^l,  
\end{align*}
where $\Lambda_h^l$ can be any of the spaces $\DPoly^l_h(\Gamma), \DjumpPoly^l_h(\Gamma)$ or $\DualPoly_h^0(\Gamma)$.

\subsection{Dirichlet boundary conditions}
Let us for the moment assume that $\Gamma\equiv\Gamma\D$.
Then, the basic idea is to add the following suitably weighted boundary residual to the weak formulation.
\begin{equation}
R_{\Gamma\D}(u_h,\lambda_h) := \betaparam\D^{1/2} (g\D-u_h).
\end{equation}
This is defined such that $R_{\Gamma\D}(u_h,\lambda_h)=0$ is equivalent to the boundary condition \cref{eq:diri_bc}.
We obtain an expression of the form
\begin{equation}\label{eq:cald_bc}
\form{C}[(u_h,\lambda_h),(v_h,\mu_h)]=\left\langle u_h,\mu_h
\right\rangle_\Gamma+\left\langle \lambda_h, v_h\right\rangle_{\Gamma} +
\left\langle R_{\Gamma\D}(u_h,\lambda_h), \betaparam_1 v_h + \betaparam_2 \mu_h \right\rangle_{\Gamma},
\end{equation}
or equivalently
\begin{equation}\label{eq:multi_bc}
\form{A}[(u_h,\lambda_h),(v_h,\mu_h)]=\tfrac12\left\langle u_h,\mu_h
\right\rangle_\Gamma+\tfrac12\left\langle \lambda_h, v_h\right\rangle_{\Gamma} +
\left\langle R_{\Gamma\D}(u_h,\lambda_h), \betaparam_1 v_h + \betaparam_2 \mu_h \right\rangle_{\Gamma},
\end{equation}
where $\betaparam_1$ and $\betaparam_2$ are problem dependent scaling operators that can be chosen as a
function of the physical parameters in order to obtain robustness of
the method.

For the Dirichlet problem, we choose $\betaparam_1 = \betaparam\D^{1/2}$, 
$\betaparam_2 = \betaparam\D^{-1/2}$, where 
different choices for
$\betaparam\D$ in the range $0\leqslant \betaparam\D \lesssim h^{-1}$ are possible.
Inserting this into \cref{eq:multi_bc}, we obtain the formulation:
\begin{multline}\label{eq:multi_D}
\form{A}[(u,\lambda),(v_h,\mu_h)] -\tfrac12 \left\langle  \lambda_h, v_h
\right\rangle_{\Gamma\D} + \tfrac12
\left\langle u_h, \mu_h \right\rangle_{\Gamma\D} +\left\langle \betaparam\D  u_h, v_h 
\right\rangle_{\Gamma\D} \\
= \left\langle g\D, \betaparam\D v_h + \mu_h
\right\rangle_{\Gamma\D}.
\end{multline}
By formally identifying $\lambda_h$ with $\partial_\nu u_h$ and $\mu_h$ with
$\partial_\nu v_h$, we obtain the classical (non-symmetric) Nitsche's method
 (up to the multiplicative factor $\tfrac12$).

For a more compact notation, we introduce the boundary operator
associated with the non-homogeneous Dirichlet condition
\begin{equation}\label{eq:operator_Dirichlet}
\form{B}\D[(u_h,\lambda_h),(v_h,\mu_h)]:=-\tfrac12\left\langle  \lambda_h, v_h \right\rangle_{\Gamma\D} +
\tfrac12\left\langle u_h, \mu_h \right\rangle_{\Gamma\D} +\left\langle \betaparam\D u_h, v_h 
\right\rangle_{\Gamma\D},
\end{equation}
the operator corresponding to the left-hand side
\begin{equation}\label{def:AD}
\form{A}\D[(u_h,\lambda_h),(v_h,\mu_h)]:=\form{A}[(u_h,\lambda_h),(v_h,\mu_h)]+ \form{B}\D
[(u_h,\lambda_h),(v_h,\mu_h)] 
\end{equation}
and the operator associated with the right-hand side
\begin{equation}\label{eq:LDir}
\form{L}\D(v_h,\mu_h) := \left\langle  g\D,\betaparam\D v_h + \mu_h\right\rangle_{\Gamma\D}.
\end{equation}

Using these and \cref{eq:multi_D}, we arrive at the following boundary element formulation: Find
$(u_h,\lambda_h) \in \productspace{V}_h$ such that
\begin{align}\label{eq:abstract_form_Dir}
\form{A}\D[(u_h,\lambda_h),(v_h,\mu_h)] &=
\form{L}\D(v_h,\mu_h)&&\forall(v_h,\mu_h) \in \productspace{V}_h.
\end{align}


\noindent We introduce the following $\form{B}\D$-norm
\[
\Bnorm[\D]{(v,\mu)} := \Vnorm{(v,\mu)} + \betaparam\D^{1/2} \Ltwonorm[\Gamma\D]{v},
\]
and summarise the properties of the bilinear form
$\form{A}\D$
in the following lemma.

\begin{lemma}[Properties of the bilinear form]\label{diri_coercive}
Let $\productspace{W}$ be a product Hilbert space for the primal and flux
variables, such that $\productspace{V} \subset \productspace{W}$.
The bilinear form has the following properties:
\begin{property}[Coercivity]
If $\betaparam\D=0$ or if there exists $\betamin>0$ (independent of $h$) such that $\betaparam\D>\betamin$, then there exists $\alpha>0$ such that $\forall (v,\mu) \in \productspace{W}$
\[\alpha \Bnorm[\D]{(v,\mu)} \leqslant \form{A}\D[(v,\mu),(v,\mu)].\]
\end{property}
\begin{property}[Continuity]
There exists
$M>0$ such that \\$\forall(w,\eta),(v,\mu) \in \productspace{W}$
\[\left|\form{A}\D[(v,\mu),(w,\eta)]\right| \leqslant M\Bnorm[\D]{(v,\mu)}\Bnorm[\D]{(w,\eta)}.
\]
\end{property}
\end{lemma}
\begin{proof}
 See \cite[Section 4.1]{BeBuScro18}.
\end{proof}

\section{Weak imposition of Signorini boundary conditions}\label{sec:contact}
Recently Chouly, Hild and Renard \cite{CH13b, CHR15} showed how
contact problems can be treated in the context of Nitsche's method.
We will here show how we may use arguments
similar to theirs in the present framework to integrate unilateral
contact seamlessly. The result is a nonlinear system to which one may
apply Newton's method or a fixed-point iteration in a straightforward
manner. We prove existence and uniqueness of solutions to the nonlinear system
and optimal order error estimates. 

For the derivation of the formulation on the contact boundary we will
first omit the Dirichlet part, letting $\Gamma=\Gamma\C$.
To impose the contact conditions, we recall the following relations, introduced by Alart
and Curnier \cite{AC91}, with $\pmpart{x} := \pm \max(0,\pm x)$.
\begin{align}
(u-\bcg) &= \negpart{(u-\bcg) - \param^{-1} (\lambda-\bcpsi)} &&\text{on }\Gamma\C,\label{eq:ac_cond1}\\ 
(\lambda-\bcpsi) &= -\pospart{\param(u-\bcg) - (\lambda-\bcpsi)} &&\text{on }\Gamma\C,\label{eq:ac_cond2}
\end{align}
for all $\param>0$.
It is straighforward \cite{CH13b} to show that each of these two conditions is
equivalent to the contact boundary conditions \cref{eq:contact_bc1,eq:contact_bc2}.

To simplify the notation, we introduce the operators
\begin{align*}
P^\param(u_h,\lambda_h) &:= \param(u_h-\bcg) - (\lambda_h-\bcpsi)&&\text{and}&P^\param_0(u_h,\lambda_h) &:= \param u_h-\lambda_h.
\end{align*}
Using \cref{eq:ac_cond1}, we arrive at the following boundary term for the contact conditions
\begin{equation}\label{eq:res_bc1}
R^1_{\Gamma\C}(u_h,\lambda_h) = (\bcg-u_h) + \param^{-1}\negpart{P^\param(u_h,\lambda_h)}.
\end{equation}
Alternatively, by using \cref{eq:ac_cond2}, we arrive at the following boundary term
\begin{equation}\label{eq:res_bc2}
R^2_{\Gamma\C}(u_h,\lambda_h) = \param^{-1}\left((\bcpsi-\lambda_h) -\pospart{P^\param(u_h,\lambda_h)}\right).
\end{equation}
By using the fact that $x=\pospart{x}+\negpart{x}$, it can be shown that \cref{eq:res_bc1,eq:res_bc2} are equal.

Substituting \cref{eq:res_bc1} into \cref{eq:multi_bc}, and using the weights $\betaparam_1 = \param$  and $\betaparam_2 = 1$, we obtain
\begin{multline}\label{eq:cont_nit}
\form{A}[(u_h,\lambda_h),(v_h,\mu_h)]
    + \tfrac12 \left\langle \mu_h,u_h \right\rangle _{\Gamma\C}
    + \left\langle \param u_h-\tfrac12\lambda_h,v_h \right\rangle _{\Gamma\C}\\
    - \left\langle \negpart{P^\param(u_h,\lambda_h)},v_h+\param^{-1}\mu_h \right\rangle _{\Gamma\C}
        = \left\langle g\C, \param v_h+\mu_h \right\rangle _{\Gamma\C}.
\end{multline}
Using \cref{eq:res_bc2}, we have
\begin{multline}\label{eq:aug_Lagr}
\form{A}[(u_h,\lambda_h),(v_h,\mu_h)]
    + \tfrac12 \left\langle \lambda_h,v_h\right\rangle _{\Gamma\C}
    + \left\langle \param^{-1}\lambda_h-\tfrac12 u_h,\mu_h \right\rangle _{\Gamma\C}\\
    + \left\langle \pospart{P^\param(u_h,\lambda_h)},v_h + \param^{-1}\mu_h\right\rangle _{\Gamma\C}
        = \left\langle \bcpsi, v_h+\param^{-1}\mu_h \right\rangle _{\Gamma\C}.
\end{multline}

We see that \cref{eq:aug_Lagr} is similar to the non-symmetric
version of the method proposed in \cite{CHR15} and \cref{eq:cont_nit}
is similar to the non-symmetric Nitsche formulation for contact
discussed in \cite{BHL16}. As pointed out in the latter reference, the two
formulations are equivalent, with the same solutions. In what follows, we
focus exclusively on the variant \cref{eq:aug_Lagr}.

\noindent Defining
\begin{align}
\begin{split}
\form{B}\C[(u_h,\lambda_h),(v_h,\mu_h)]&:=\tfrac12 \left\langle \lambda_h,v_h\right\rangle _{\Gamma\C}
    + \left\langle \param^{-1}\lambda_h-\tfrac12 u_h,\mu_h \right\rangle _{\Gamma\C}\label{def:BC}
   \\
    &\qquad\quad+ \left\langle \pospart{P^\param(u_h,\lambda_h)},v_h + \param^{-1}\mu_h\right\rangle _{\Gamma\C}, \end{split}\\
\form{L}\C(v_h,\mu_h)&:=\left\langle \bcpsi, v_h+\param^{-1}\mu_h \right\rangle _{\Gamma\C},\label{def:LC}\\
\form{A}\C[(u_h,\lambda_h),(v_h,\mu_h)]&:=\form{A}[(u_h,\lambda_h),(v_h,\mu_h)]+\form{B}\C[(u_h,\lambda_h),(v_h,\mu_h)],
\end{align}
we arrive at the boundary element method formulation: Find $(u_h,\lambda_h) \in \productspace{V}_h$ such that
\begin{align}
\form{A}\C[(u_h,\lambda_h),(v_h,\mu_h)] &= \form{L}\C (v_h,\mu_h) &&\forall(v_h,\mu_h)\in\productspace{V}_h.\label{nonlinearsystem}
\end{align}

\subsection{Mixed Dirichlet and contact boundary conditions}
Combining the formulations for the Dirichlet and contact conditions, we arrive at the following
boundary element method for the problem \cref{eq:Laplace}:
Find $(u_h,\lambda_h) \in \productspace{V}_h$ such that
\begin{multline}\label{eq:aug_BEM}
\form{A}\D[(u_h,\lambda_h),(v_h,\mu_h)] + \form{B}\C
[(u_h,\lambda_h),(v_h,\mu_h)]
= \mathcal{L}\D (v_h,\mu_h) + \mathcal{L}\C (v_h,\mu_h) 
\\\forall (v_h,\mu_h) \in \productspace{V}_h,
\end{multline}
where $\form{A}\D$, $\form{L}\D$, $\form{B}\C$ and $\form{L}\C$ are defined in \cref{def:AD,def:BC,eq:LDir,def:LC}.
For discretisation, we use the assumptions and spaces introduced in \cref{sec:Dirichlet}.
Note that the formulation \eqref{eq:aug_BEM} is consistent, i.e.\,the continuous solution $(u,\lambda)$
to \eqref{eq:Laplace}
fulfills \eqref{eq:aug_BEM} for all $(v_h,\mu_h)\in \productspace{V}_h$.

\section{Analysis}\label{sec:analysis}
In this section, we prove the existence of unique solutions to the nonlinear
system of equations \cref{eq:aug_BEM} as well as optimal error estimates.

We assume that the solution $(u,\lambda)$ of \cref{eq:Laplace} lies in
$\productspace{W} := H^{1+\epsilon}(\Gamma)\times H^{\epsilon}(\tilde{\Gamma})$ for some $\epsilon \in (0,1/2]$,
where $\tilde{\Gamma}=\cup_{i=1}^M \Gamma_i\setminus\partial\Gamma_i$ is the set of boundary points that lie in the interior of the faces $\Gamma_i$.
As the normal vectors $\B\nu_{\x}$ are discontinuous between faces, we can not expect a higher global regularity for $\lambda$.

We define the distance function $\distance[\C]{}{}$ and norm $\starnorm\cdot$, for $(v,\mu),(w,\eta)\in\productspace{W}$, by
\begin{align}
\distance[\C]{(v,\mu)}{(w,\eta)}
&:=
\Bnorm[\D]{(v-w,\mu-\eta)}
\notag\\&\hspace{7mm}
+\Ltwonorm[\Gamma\C]{\param^{-\frac12}\left(\mu-\eta +\pospart{P^\param(v,\mu)}-\pospart{P^\param(w,\eta)}\right)},
\\
\starnorm{(v,\mu)}&:= \Bnorm[\D]{(v,\mu)}+\Ltwonorm[\Gamma\C]{\param^{\frac12}v} + \Ltwonorm[\Gamma\C]{\param^{-\frac12}\mu}.
\end{align}
We note that due the appearance of $\pospart\cdot$ in its second term, $\distance[\C]{}{}$ is not a norm.
$\distance[\C]{}{}$ does provide a bound on the error however, as for all $(v,\mu)\in\productspace{W}$,
$\distance[\C]{(v,\mu)}{(0,0)}\geqslant\Bnorm[\D]{(v,\mu)}\geqslant\Vnorm{(v,\mu)}$.

When proving this section's results, we will use properties of the $\pospart\cdot$ function that are given in the following lemma.
\begin{lemma}\label{lem:monotone}
For all $a,b \in \mathbb{R}$,
\begin{align}
\left(\pospart{a}-\pospart{b}\right)^2 &\leqslant \left(\pospart{a} - \pospart{b}\right) (a - b),\label{eq:tech1}\\
\seminorm{\pospart{a}-\pospart{b}} &\leqslant \seminorm{a-b}.\label{eq:tech2}
\end{align}
\end{lemma}
\begin{proof}
For a proof of these well-known properties see e.g.~\cite{CH13b}.
\end{proof}

We now prove a result analogous to the coercivity assumption in \cite{BeBuScro18}.

\begin{lemma}\label{prop:monotone}
If there is $\betamin>0$, independent of $h$, such that $\betaparam\D>\betamin$, then
there is $\alpha>0$ such that for all $(v,\mu),(w,\eta) \in \productspace{W}$,
\begin{multline*}
\alpha\left(\distance[\C]{(v,\mu)}{(w,\eta)}\right)^2
\leqslant (\form{A}+\form{B}\D)[(v-w,\mu-\eta),(v-w,\mu-\eta)]\\
+  \form{B}\C [(v,\mu),(v-w,\mu-\eta)] - \form{B}\C [(w,\eta),(v-w,\mu-\eta)].
\end{multline*}
\end{lemma}
\begin{proof}
From the analysis of the Dirichlet problem (\cref{diri_coercive}) we know that when $\betaparam\D>\betamin>0$,
\begin{equation}\label{eq:pmono1}
\alpha \Bnorm[\D]{(v-w,\mu-\eta)}^2\leqslant  (\form{A}+\form{B}\D)[(v-w,\mu-\eta),(v-w,\mu-\eta)].
\end{equation}
Introducing the notation 
$\changeP:= \pospart{P^\param(v,\mu)}-\pospart{P^\param(w,\eta)}$, we have
\begin{multline}
 \form{B}\C[(v,\mu),(v-w,\mu-\eta)] - \form{B}\C[(w,\eta),(v-w,\mu-\eta)]\\
 =\param^{-1} \Ltwonorm[\Gamma\C]{\mu-\eta}^2 + \Ltwoinner[\Gamma\C]{\changeP}{v-w+\param^{-1}(\mu-\eta)}.
\end{multline}
To estimate the expression on the right-hand side, we use
\begin{equation*}
 \param^{-1}
\Ltwonorm[\Gamma\C]{\mu-\eta +\changeP}^2
= \param^{-1} \left( \Ltwonorm[\Gamma\C]{\mu-\eta}^2 + \Ltwonorm[\Gamma\C]{\changeP}^2
+ 2 \Ltwoinner[\Gamma\C]{\mu-\eta}{\changeP}\right).
\end{equation*}
Using \cref{eq:tech1}, this implies
the bound
\begin{multline*}
 \param^{-1}
\Ltwonorm[\Gamma\C]{\mu-\eta+\changeP}^2
\\\leqslant 
\param^{-1}\left(
\Ltwonorm[\Gamma\C]{\mu-\eta}^2 + \Ltwoinner[\Gamma\C]{\changeP}{P_0^\param(v-w,\mu-\eta)}
+ 2\Ltwoinner[\Gamma\C]{\mu-\eta}{\changeP}\right).
\end{multline*}

Observing that $P_0^\param(v-w,\mu-\eta) + 2(\mu-\eta)= \param(v-w) + \mu-\eta$, we infer that
\begin{equation}\label{signor:lemmaproof2}
\param^{-1}
\Ltwonorm[\Gamma\C]{\mu-\eta+ \changeP}^2
\leqslant \form{B}\C [(v,\mu),(v-w,\mu-\eta) - \form{B}\C [(w,\eta),(v-w,\mu-\eta)].
\end{equation}
We conclude the proof by noting that
\begin{multline*}
\left(\distance[\C]{(v,\mu)}{(w,\eta)}\right)^2
\lesssim
\Bnorm[\D]{(v-w,\mu-\eta)}^2\\+\param^{-1}\Ltwonorm[\Gamma\C]{\mu-\eta +\pospart{P^\param(v,\mu)}-\pospart{P^\param(w,\eta)}}^2,
\end{multline*}
and applying \cref{eq:pmono1,signor:lemmaproof2}.
\end{proof}

Next, we prove a result analagous to the discrete coercivity assumption in \cite{BeBuScro18}.
\begin{lemma}\label{corol:aug_pos}
If there is $\betamin>0$, independent of $h$, such that $\betaparam\D>\betamin$, then
there is $\alpha>0$ such that for all $(v_h,\mu_h) \in \productspace{V}_h$,
\begin{multline*}
\alpha\left(\Bnorm[\D]{(v_h,\mu_h)}+\Ltwonorm[\Gamma\C]{\param^{-\frac12}\left(\mu_h +\pospart{P^\param(v_h,\mu_h)}\right)}\right)^2\\
\leqslant (\form{A}+\form{B}\D+\form{B}\C)[(v_h,\mu_h),(v_h,\mu_h)]
- \Ltwoinner[\Gamma\C]{\pospart{P^\param(v_h,\mu_h)}}{g\C - \param^{-1} \psi\C}
\end{multline*}
\end{lemma}
\begin{proof}
The proof is similar to that of \cref{prop:monotone}, but
with $\mu_h$ and $v_h$ instead of $\mu-\eta$ and $v-w$. The appearance of the data term in the right-hand side is due to
the relation
\begin{align*}
\param^{-1} \Ltwonorm[\Gamma\C]{\pospart{P^\param(v_h,\mu_h)}}^2
\hspace{-20mm}&\hspace{20mm}
+ 2 \param^{-1} \Ltwoinner[\Gamma\C]{\mu_h}{\pospart{P^\param(v_h,\mu_h)}}
+ \param^{-1} \Ltwonorm[\Gamma\C]{\mu_h}^2\\
&= \param^{-1} \Ltwoinner[\Gamma\C]{\pospart{P^\param(v_h,\mu_h)}}{P^\param(v_h,\mu_h)}
+\param^{-1} \Ltwonorm[\Gamma\C]{\mu_h}^2\\
&= \Ltwoinner[\Gamma\C]{\pospart{P^\param(v_h,\mu_h)}}{u_h + \param^{-1} \mu_h}
\\&\hspace{8mm}
- \Ltwoinner[\Gamma\C]{\pospart{P^\param(v_h,\mu_h)}}{g\C - \param^{-1} \psi\C}
+ \param^{-1} \Ltwonorm[\Gamma\C]{\mu_h}^2\\
&= \form{B}\C [(v_h,\mu_h),(v_h,\mu_h)]
- \Ltwoinner[\Gamma\C]{\pospart{P^\param(v_h,\mu_h)}}{g\C - \param^{-1} \psi\C}.
\end{align*}
\end{proof}

\noindent Using \cref{prop:monotone,corol:aug_pos}, we may now prove that \cref{eq:aug_BEM} is well-posed.
\begin{theorem}
The finite dimensional nonlinear system \cref{eq:aug_BEM} admits a unique solution.
\end{theorem}
\begin{proof}
To prove the existence of a solution, we show the continuity and the positivity of the
nonlinear operator $\form{A}+\form{B}\D+\form{B}\C$.
This allows us to apply Brouwer's fixed point theorem, see eg \cite[Chapter 2, Lemma 1.4]{Te77}. 

We define $\bop{F}:\productspace{V}_h\to\productspace{V}_h$, for $(v_h,\mu_h)\in\productspace{V}_h$, by
\begin{multline*}
\Ltwoinner{\bop{F}(v_h, \mu_h)}{(w_h,\eta_h)}
= (\form{A}+\form{B}\D+\form{B}\C)[(v_h,\mu_h),(w_h,\eta_h)]
\\- \form{L}\D (w_h,\eta_h)- \form{L}\C (w_h,\eta_h),
\end{multline*}
for all $(w_h,\eta_h)\in\productspace{V}_h$. We may write the non-linear system \cref{eq:aug_BEM} as
\begin{align}
\Ltwoinner{\bop{F}(v_h, \mu_h)}{(w_h,\eta_h)}&=0&&\forall(w_h,\eta_h)\in\productspace{V}_h.\label{signoproof:form}
\end{align}
For fixed $h$, by the equivalance of norms on discrete spaces, there exist $c_1,c_2>0$ such that for all $(v_h,\mu_h)\in\productspace{V}_h$,
\[
c_1 \Ltwonorm{(v_h,\mu_h)}
\leqslant
\Bnorm[\D]{(v_h,\mu_h)}
\leqslant
c_2 \Ltwonorm{(v_h,\mu_h)}.
\]

\noindent To show positivity, we let $(v_h,\mu_h)\in\productspace{V}_h$. Using \cref{corol:aug_pos}, we see that
\begin{multline*}
\Ltwoinner{\bop{F}(v_h,\mu_h)}{(v_h,\mu_h)}
    \geqslant \alpha \Bnorm[\D]{(v_h,\mu_h)}^2
    + \alpha\param^{-1}\Ltwonorm[\Gamma\C]{\mu_h + \pospart{P^\param(v_h,\mu_h)}}^2 \\
    + \Ltwoinner[\Gamma\C]{ \pospart{P^\param(v_h,\mu_h)}}{g\C - \param^{-1} \psi\C }
    - \form{L}\D (v_h,\mu_h) - \form{L}\C (v_h,\mu_h).
\end{multline*}
Using the Cauchy--Schwarz inequality and
an arithmetic-geometric inequality, we see that there exists $C_{g\C\psi\C}>0$ such that
\begin{align*}
&\Ltwoinner[\Gamma\C]{ \pospart{P^\param(v_h,\mu_h)}}{ g\C - \param^{-1} \psi\C }
    - \form{L}\D (v_h,\mu_h) - \form{L}\C (v_h,\mu_h)\\
&\hspace{20mm}=
\Ltwoinner[\Gamma\C]{ \pospart{P^\param(v_h,\mu_h)} + \mu_h}{g\C - \param^{-1} \psi\C }
    - \Ltwoinner[\Gamma\C]{ \mu_h}{g\C - \param^{-1} \psi\C}
\\&\hspace{28mm}
-\Ltwoinner[\Gamma\D]{ g\D}{ \betaparam\D v_h +\mu_h}
    - \Ltwoinner[\Gamma\C]{ \psi\C}{ v_h + \param^{-1}\mu_h}
   \\
&\hspace{20mm}\geqslant
    -C_{g\C\psi\C}^2 - \tfrac\alpha2 \left(\Bnorm[\D]{(v_h,\mu_h)}^2 + \param^{-1}
\Ltwonorm[\Gamma\C]{\mu_h + \pospart{P^\param(v_h,\mu_h)}}^2\right).
\end{align*}
Using norm equivalence, we obtain
\begin{align*}
&\Ltwoinner{\bop{F}(v_h,\mu_h)}{(v_h,\mu_h)}\\
    &\hspace{20mm}\geqslant\tfrac\alpha2\left(\Bnorm[\D]{(v_h,\mu_h)}^2 + \param^{-1}\Ltwonorm[\Gamma\C]{\mu_h + \pospart{P^\param(v_h,\mu_h)}}^2\right)-C_{g\C\psi\C}^2\\
    &\hspace{20mm}\geqslant C'\Ltwonorm{(v_h,\mu_h)}^2 - C_{g\C\psi\C}^2,
\end{align*}
for some $C'>0$. We conclude that for all $(v_h,\mu_h) \in\productspace{V}_h$ with
\[
\Ltwonorm{(v_h,\mu_h)}^2 > \frac{C_{g\C\psi\C}^2}{C'}+1,
\]
there holds $\Ltwoinner{\bop{F}(v_h,\mu_h)}{(v_h,\mu_h)}> 0$. 

To show continuity, let $(v^1_h,\mu^1_h),(v^2_h,\mu^2_h)\in\productspace{V}_h$.
We have for all $(w_h,\eta_h)\in\productspace{V}_h$,
\begin{align*}
\hspace{20mm}&\hspace{-20mm}
\Ltwoinner{\bop{F}(v^1_h,\mu^1_h)-\bop{F}(v^2_h,\mu^2_h)}{(w_h,\eta_h)}\\
&=
\Ltwoinner[\Gamma\C]{\pospart{P^\param(v^1_h,\mu^1_h)}-\pospart{P^\param(v^2_h,\mu^2_h)}}{w_h + \param^{-1} \eta_h}
\\&\hspace{10mm}
    +\tfrac12\Ltwoinner{\mu_h^1-\mu_h^2}{w_h+\param^{-1} \eta_h} 
    - \tfrac12\Ltwoinner[\Gamma\C]{v^1_h-v^2_h}{\mu^1_h-\mu^2_h}
\\&\hspace{10mm}
    + (\form{A}+\form{B}\D)[(v^1_h-v^2_h,\mu^1_h-\mu^2_h), (w_h,\eta_h)]\\
&\leqslant \left(\param\Ltwonorm[\Gamma\C]{v^1_h-v^2_h} + \Ltwonorm[\Gamma\C]{\mu^1_h-\mu^2_h}\right)
\left(\Ltwonorm[\Gamma\C]{w_h} + \param^{-1} \Ltwonorm[\Gamma\C]{\eta_h}\right),
\end{align*}
where we have used \cref{eq:tech2}. By norm equivalence, this means that
\[
\frac{\Ltwoinner{\bop{F}(v^1_h,\mu^1_h)-\bop{F}(v^2_h,\mu^2_h)}{(w_h,\eta_h)}}{\Ltwonorm{(w_h,\eta_h)}}\leqslant
C\Ltwonorm{(v^1_h-v^2_h,\mu^1_h-\mu^2_h)}
\]
showing that $\bop{F}$ is continuous.

It then follows by Brouwer's fixed point theorem \cite[Chapter 2, Lemma 1.4]{Te77}
that there exists a solution to \cref{signoproof:form} and hence also to \cref{eq:aug_BEM}.

Uniqueness is an immediate consequence of \cref{prop:monotone}. Assume that $(u_h^1,\lambda_h^1)$ and
$(u_h^2,\lambda_h^2)$ are solutions to \cref{eq:aug_BEM}. We immediately see that
\[
\alpha\left(\distance[\C]{(u_h^1,\lambda_h^1)}{(u_h^2,\lambda_h^2)}\right)^2 = 0,
\]
and we conclude that the solution is unique.
\end{proof}

We now proceed to prove the following best approximation result.
\begin{lemma}\label{lem:normlemma}
Let $(u,\lambda)\in \productspace{W}$ be the solution of \cref{eq:Laplace} and $(u_h,\lambda_h) \in \productspace{V}_h$ the
solution of \cref{eq:aug_BEM}. Then there holds
\[
\distance[\C]{(u,\lambda)}{(u_h,\lambda_h)} \leqslant C
  \inf_{(v_h,\mu_h)\in \productspace{V}_h} \starnorm{(u-v_h,\lambda - \mu_h)}.
\]
\end{lemma}
\begin{proof}
Using \cref{prop:monotone} and Galerkin orthogonality, we see that, 
for arbitrary $(v_h,\mu_h) \in \productspace{V}_h$,
\begin{align*}
\alpha\hspace{5mm}&\hspace{-5mm}\left(\distance[\C]{(u,\lambda)}{(u_h,\lambda_h)}\right)^2\\
&\leqslant (\form{A}+\form{B}\D)[(u-u_h,\lambda-\lambda_h),(u-u_h,\lambda-\lambda_h)] 
\\&\hspace{8mm}
+  \form{B}\C [(u,\lambda),(u-u_h,\lambda-\lambda_h)]
- \form{B}\C [(u_h,\lambda_h),(u-u_h,\lambda-\lambda_h)]\\
&= (\form{A}+\form{B}\D)[(u-u_h,\lambda-\lambda_h),(u-v_h,\lambda-\mu_h)]
\\&\hspace{8mm}
+\form{B}\C [(u,\lambda),(u-v_h,\lambda-\mu_h)]
-\form{B}\C [(u_h,\lambda_h),(u-v_h,\lambda-\mu_h)].
\end{align*}
Next, we use
\begin{multline*}
\form{B}\C [(u,\lambda),(u-v_h,\lambda-\mu_h)] - \form{B}\C [(u_h,\lambda_h),(u-v_h,\lambda-\mu_h)]\\
= \Ltwoinner[\Gamma\C]%
    {\lambda-\lambda_h +\pospart{P^\param(u,\lambda)}-\pospart{P^\param(u_h,\lambda_h)}}%
    {(u-v_h)+\param^{-1} (\lambda - \mu_h)}\\
-\tfrac12 \Ltwoinner[\Gamma\C]{u-u_h}{\lambda -\mu_h}
-\tfrac12 \Ltwoinner[\Gamma\C]{\lambda-\lambda_h}{u -v_h}
\end{multline*}
to show that
\begin{multline*}
(\form{A}+\form{B}\D)[(u-u_h,\lambda-\lambda_h),(u-u_h,\lambda-\lambda_h)] 
\\+\form{B}\C [(u,\lambda),(u-u_h,\lambda-\lambda_h)] - \form{B}\C [(u_h,\lambda_h),(u-u_h,\lambda-\lambda_h)] \\
=\underbrace{ (\form{A}+\form{B}\D)[(u-u_h,\lambda-\lambda_h),(u-v_h,\lambda - \mu_h)]}_\text{(I)}\\
\qquad\underbrace{\null-\tfrac12 \Ltwoinner[\Gamma\C]{ u-u_h}{ \lambda -\mu_h} 
- \tfrac12 \Ltwoinner[\Gamma\C]{ \lambda-\lambda_h}{ u -v_h}}_\text{(II)}\\
\quad+ \underbrace{\Ltwoinner[\Gamma\C]{ \lambda-\lambda_h +\pospart{P^\param(u,\lambda)} -\pospart{P^\param(u_h,\lambda_h)}}{(u-v_h)+\param^{-1} (\lambda - \mu_h)}}_\text{(III)}.
\end{multline*}

We estimate the three parts of the right-hand separately.
For the first term, we use the continuity of $\form{A}+\form{B}\D$ (\cref{diri_coercive}) to obtain
\begin{equation*}
\text{(I)}\leqslant M \Bnorm[\D]{(u-u_h,\lambda-\lambda_h)}  \Bnorm[\D]{(u-v_h,\lambda - \mu_h)}.
\end{equation*}
For the second line, we use $\Hspace{1/2}$--$\Hspace{-1/2}$ duality and the
  Cauchy--Schwarz inequality to obtain
\begin{equation*}
\text{(II)} \leqslant \Bnorm[\D]{(u-u_h,\lambda-\lambda_h)}  \Bnorm[\D]{(u-v_h,\lambda - \mu_h)}.
\end{equation*}
For the last term, we use the Cauchy--Schwarz inequality to get
\begin{multline*}
\text{(III)} \leqslant \Ltwonorm[\Gamma\C]{\param^{-1/2}\left(\lambda-\lambda_h +\pospart{P^\param(u,\lambda)}-
  \pospart{P^\param(u_h,\lambda_h)}\right)}\\
\cdot \left(\Ltwonorm[\Gamma\C]{\param^{1/2}(u-v_h)} + \Ltwonorm[\Gamma\C]{\param^{-1/2}(\lambda-\mu_h)}\right).
\end{multline*}
Collecting these bounds, we see that
\begin{equation*}
\distance[\C]{(u,\lambda)}{(u_h,\lambda_h)}^2 \lesssim
\distance[\C]{(u,\lambda)}{(u_h,\lambda_h)}\starnorm{(u-v_h,\lambda - \mu_h)}.
\end{equation*} Dividing through by $\distance[\C]{(u,\lambda)}{(u_h,\lambda_h)}$, and taking the infimum yields the desired result.
\end{proof}

We now prove the main result of this section, an \textit{a priori} bound on the error of the solution of \cref{eq:aug_BEM}.
\begin{theorem}\label{main_contact_result}
Let $(u,\lambda)\in \Hspace{s}\times\Hspace[\tilde\Gamma]{r}$ 
for some $s\geqslant1, r\geqslant0$ and $(u_h,\lambda_h)\in\Poly_h^k(\Gamma)\times\Lambda_h^l$ be the solutions of \cref{eq:Laplace}
and the discrete problem \cref{eq:aug_BEM}, respectively. If there is $\betamin>0$ such that $\betamin<\betaparam\D\lesssim h^{-1}$ and $\param \eqsim h^{-1}$, 
then
\begin{align*}
\Vnorm{(u-u_h,\lambda-\lambda_h)} &\leqslant
\distance[\C]{(u,\lambda)}{(u_h,\lambda_h)} \\
&\lesssim
h^{\zeta-1/2} \Hseminorm{\zeta}{u} + h^{\xi+1/2}
\Hseminorm[\tilde\Gamma]{\xi}{\lambda},
\end{align*}
where $\zeta = \min(k+1,s)$ and $\xi = \min(l+1,r)$ for $\Lambda_h^l\in \{\DPoly^l_h(\Gamma), \DjumpPoly^l_h(\Gamma)\}$
and $\zeta = \min(2,s)$ and $\xi = \min(\frac{1}{2},r)$ for $\Lambda_h^l=\DualPoly^0_h(\Gamma)$. Additionally,
\[
\Hnorm[\Omega]{1}{\tilde u-\tilde u_h} \lesssim
h^{\zeta-1/2} \Hseminorm{\zeta}{u} + h^{\xi+1/2}
\Hseminorm[\tilde\Gamma]{\xi}{\lambda},
\]
where $\tilde u$ and $\tilde u_h$ are the solutions in $\Omega$ defined by \cref{eq:represent}.
\end{theorem}
\begin{proof}
First, we observe that for all $(v,\mu)$ and $(w,\eta)$ in $\productspace{W}$
\[
 \Vnorm{(v-w,\mu-\eta)}\leqslant\distance[\C]{(v,\mu)}{(w,\eta)}.
 \]
Using standard approximation results
for $\Lambda_h^l\in \{\DPoly^l_h(\Gamma), \DjumpPoly^l_h(\Gamma)\}$ (see eg \cite[chapter 10]{Stein07})
and \cref{lemma:DUAL0app} for $\Lambda_h^l=\DualPoly^0_h(\Gamma)$, we see that
\begin{align*}
\inf_{(v_h,\mu_h)\in\productspace{V}_h}\hspace{-2mm}\Vnorm{(u-v_h,\lambda-\mu_h)}
&=\inf_{v_h\in \Poly^k_h(\Gamma)}\hspace{-5mm}\Hnorm{1/2}{u-v_h}
+\inf_{\mu_h\in\Lambda^l_h(\Gamma)}\hspace{-5mm}\Hnorm{-1/2}{\lambda-\mu_h}\\
&\lesssim h^{\zeta-1/2}\Hseminorm{\zeta}{u}
         + h^{\xi+1/2}\Hseminorm[\tilde\Gamma]{\xi}{\lambda},\\
\inf_{v_h\in\Poly^k_h(\Gamma)}\Ltwonorm[\Gamma]{u-v_h}&\lesssim h^{\zeta}\Hseminorm{\zeta}{u},\quad
\inf_{\mu_h\in\Lambda^l_h}\Ltwonorm[\Gamma]{\lambda-\mu_h}\lesssim h^{\xi}\Hseminorm[\tilde\Gamma]{\xi}{\lambda}.
\end{align*}
Applying these to the definition of $\starnorm{\cdot}$ gives
\begin{multline*}
\inf_{(v_h,\mu_h)\in\productspace{V}_h}\starnorm{(u-v_h,\lambda-\mu_h)}
\lesssim h^{\zeta-1/2}\Hseminorm{\zeta}{u}
         + h^{\xi+1/2}\Hseminorm[\tilde\Gamma]{\xi}{\lambda}\\
         + \betaparam\D^{1/2}h^{\zeta}\Hseminorm{\zeta}{u}
         + \param^{1/2} h^{\zeta}\Hseminorm{\zeta}{u}
         +\param^{-1/2} h^{\xi}\Hseminorm[\tilde\Gamma]{\xi}{\lambda}.   
\end{multline*}
By means of \cref{lem:normlemma} and the given choice of the parameters $\param$ and $\betaparam\D$ this proves the first assertion.
The estimate in the domain $\Omega$ follows by using the relations
\cref{eq:sing_rec_stab,eq:doub_rec_stab}.
\end{proof}
If $\lambda$ is smooth enough and $k=l$, the bounds on $\param$ can be replaced with $h\lesssim\param\lesssim h^{-1}$ without reducing
the order of convergence.

\section{Numerical results}\label{sec:numerical}
We now demonstrate the theory with a series of numerical examples. In this section, we consider the following test problem.
Let $\Omega=[0,1]\times[0,1]\times[0,1]$ be the unit cube, $\Gamma\C:=\{(x,y,z)\in\Gamma:z=1\}$, and
$\Gamma\D:=\Gamma\setminus\Gamma\C$.
Let
\begin{subequations}\label{sig:testbc2}
\begin{align}
 g_D&=0,\\
 g\C&=
    \begin{cases}
        \sin(\pi x)\sin(\pi y)\sinh(\sqrt2\pi)&x\leqslant\frac12\\
        \sin(\pi y)\sinh(\sqrt2\pi)&x>\frac12
    \end{cases},\\
 \psi\C&=
    \begin{cases}
        \sqrt2\pi\sin(\pi x)\sin(\pi y)\cosh(\sqrt2\pi)&x\geqslant\frac12\\
        \sqrt2\pi\sin(\pi y)\cosh(\sqrt2\pi)&x<\frac12
    \end{cases}.
\end{align}
\end{subequations}
It can be shown that
\begin{equation*}
 u(x,y,z)=
    \sin(\pi x)\sin(\pi y)\sinh(\sqrt2\pi z)
\end{equation*}
is the solution to \cref{eq:Laplace} with these boundary conditions.

To solve the non-linear system \cref{nonlinearsystem}, we will treat the nonlinear term explicitly. Therefore, we
define
\begin{align}
\form{B}\C'[(u,\lambda),(v,\mu)]&:=\tfrac12 \Ltwoinner[\Gamma\C]{ \lambda}{v}
    + \Ltwoinner[\Gamma\C]{ \param^{-1}\lambda-\tfrac12 u}{\mu }
\end{align}
Note that $\form{B}'\C$ differs from $\form{B}\C$ only by the missing nonlinear term.

We pick initial guesses $(u_0,\lambda_0)\in\productspace{V}_h$ and define $(u_{n+1},\lambda_{n+1})\in\productspace{V}_h$,
for $n\in\mathbb{N}$, to be the solution of
\begin{multline}
(\form{A}+\form{B}\D+\form{B}\C')[(u_{n+1},\lambda_{n+1}),(v_h,\mu_h)] \\= \form{L}\C(v_h,\mu_h)-\Ltwoinner[\Gamma\C]{ \pospart{P^\param(u_n,\lambda_n)}}{v_h + \param^{-1}\mu_h}
\quad\forall(v_h,\mu_h)\in\productspace{V}_h.\label{eq:alg_form}
\end{multline}
This leads us to \cref{thealgorithm}, an iterative method for solving the contact problem.

In all the computations in this section, we preconditioned the GMRES solver using a 
mass matrix preconditioner applied blockwise from the left, as described in \cite{Betcke17}.

\begin{algorithm}
  \caption{Iterative algorithm for solving the contact problem}\label{thealgorithm}
\begin{algorithmic}
\STATE{Input $(u_0,\lambda_0)$, \textsc{tol}, \textsc{maxiter}}
\FOR{$n\gets0$ to \textsc{maxiter}}
\STATE{$(u_{n+1},\lambda_{n+1})\gets\text{solution of \cref{eq:alg_form}, calculated using GMRES}$}
\IF{$\Vnorm{(u_{n+1},\lambda_{n+1})-(u_{n},\lambda_{n})}< \textsc{tol}$}
\RETURN{$(u_{n+1},\lambda_{n+1})$}
\ENDIF
\ENDFOR
\end{algorithmic}
\end{algorithm}

\begin{figure}
\centering
\setlength\tabcolsep{0pt}
\begin{tabular}{rr}
\begin{tikzpicture}
\begin{axis}[small,axis on top,
             xlabel={$\param$},xmode=log,xmin=1e-2,xmax=1e2,
             ylabel near ticks,ylabel={Error in $\productspace{V}$ norm},ymode=log,ymax=1e2,
]
\DConeplot table {img/cube/old_dual_param_2_er.dat};
\DCtwoplot table {img/cube/old_dual_param_3.5_er.dat};
\DCthreeplot table {img/cube/old_dual_param_5_er.dat};
\end{axis}
\end{tikzpicture}&
\begin{tikzpicture}
\begin{axis}[small,axis on top,
             xlabel={$\param$},xmode=log,xmin=1e-2,xmax=1e2,
             ylabel near ticks,ylabel={Error in $\productspace{V}$ norm},ymode=log,ymax=1e2,
]
\Coneplot table {img/cube/old_P1DP0_param_2_er.dat};
\Ctwoplot table {img/cube/old_P1DP0_param_3.5_er.dat};
\Cthreeplot table {img/cube/old_P1DP0_param_5_er.dat};
\end{axis}
\end{tikzpicture}\\
\begin{tikzpicture}
\begin{axis}[small,axis on top,
             xlabel={$\param$},xmode=log,xmin=1e-2,xmax=1e2,
             ylabel near ticks,ylabel={{\textnumero} of outer iterations},ymin=0,ymax=50
]
\DConeplot table {img/cube/old_dual_param_2_it.dat};
\DCtwoplot table {img/cube/old_dual_param_3.5_it.dat};
\DCthreeplot table {img/cube/old_dual_param_5_it.dat};
\end{axis}
\end{tikzpicture}
&
\begin{tikzpicture}
\begin{axis}[small,axis on top,
             xlabel={$\param$},xmode=log,xmin=1e-2,xmax=1e2,
             ylabel near ticks,ylabel={{\textnumero} of outer iterations},ymin=0,ymax=50
]
\Coneplot table {img/cube/old_P1DP0_param_2_it.dat};
\Ctwoplot table {img/cube/old_P1DP0_param_3.5_it.dat};
\Cthreeplot table {img/cube/old_P1DP0_param_5_it.dat};
\end{axis}
\end{tikzpicture}\\
\begin{tikzpicture}
\begin{axis}[small,axis on top,
             xlabel={$\param$},xmode=log,xmin=1e-2,xmax=1e2,
             ylabel near ticks,ylabel={\tiny Average {\textnumero} of GMRES iterations},ylabel style={align=center},ymin=0,ymax=300
]
\DConeplot table {img/cube/old_dual_param_2_at.dat};
\DCtwoplot table {img/cube/old_dual_param_3.5_at.dat};
\DCthreeplot table {img/cube/old_dual_param_5_at.dat};
\end{axis}
\end{tikzpicture}&
\begin{tikzpicture}
\begin{axis}[small,axis on top,
             xlabel={$\param$},xmode=log,xmin=1e-2,xmax=1e2,
             ylabel near ticks,ylabel={\tiny Average {\textnumero} of GMRES iterations},
             ylabel style={align=center},ymin=0,ymax=2000
]
\Coneplot table {img/cube/old_P1DP0_param_2_at.dat};
\Ctwoplot table {img/cube/old_P1DP0_param_3.5_at.dat};
\Cthreeplot table {img/cube/old_P1DP0_param_5_at.dat};
\end{axis}
\end{tikzpicture}
\end{tabular}
\caption{The dependence of the error, number of outer iterations, and the average number of GMRES iterations on $\param$,
for the problem \cref{eq:Laplace} with boundary conditions \cref{sig:testbc2} on the unit cube with
$h=2^{-2}$ (\onedesc), $h=2^{-3.5}$ (\twodesc), and $h=2^{-5}$ (\threedesc). Here we take
$u_0=\lambda_0=0$, $\betaparam\D=0.01$, $\textsc{tol}=0.05$, and $\textsc{maxiter}=50$.
On the left (\DCcolordesc), we take $(u_n,\lambda_n),(v_h,\mu_h)\in\Poly_h^1(\Gamma)\times\DualPoly_h^0(\Gamma)$;
on the right (\Ccolordesc), we take $(u_n,\lambda_n),(v_h,\mu_h)\in\Poly_h^1(\Gamma)\times\DPoly_h^0(\Gamma)$.
}
\label{figsig:cube_dual_param}
\end{figure}

Inspired by the parameter choices in \cite{BeBuScro18}, we fix $\betaparam\D=0.01$ and look for suitable values of the parameter $\param$.
\cref{figsig:cube_dual_param} shows how the error, number of outer iterations, and the average number of GMRES
iterations inside each outer iteration change as the parameter $\param$ is varied,
for both $\productspace{V}_h=\Poly_h^1(\Gamma)\times\DualPoly_h^0(\Gamma)$ (left, \DCcolordesc) and $\productspace{V}_h=\Poly_h^1(\Gamma)\times\DPoly_h^0(\Gamma)$ (right, \Ccolordesc).
Here, we see that the error and number of outer iterations are lowest when $\param$ is between around 1 and 10.

\begin{figure}
\centering
\centering
\setlength\tabcolsep{0pt}
\begin{tabular}{rr}
\begin{tikzpicture}
\begin{axis}[small,axis on top,axis equal,
             xmax=1,xmin=0.01,
             xlabel={$h$},xmode=log,x dir=reverse,
             ylabel near ticks,ylabel={Error in $\productspace{V}$ norm},ymode=log,
]
\referenceplot table {
0.4 10
0.03 0.75
};
\DCplot table {img/cube/old_dual_conv_er.dat};
\end{axis}
\end{tikzpicture}
&
\begin{tikzpicture}
\begin{axis}[small,axis on top,axis equal,
             xmax=1,xmin=0.01,
             xlabel={$h$},xmode=log,x dir=reverse,
             ylabel near ticks,ylabel={Error in $\productspace{V}$ norm},ymode=log,
]
\referenceplot table {
0.2 20
0.04 1.78885438
};
\Cplot table {img/cube/old_P1DP0_conv_er.dat};
\end{axis}
\end{tikzpicture}
\\
\begin{tikzpicture}
\begin{axis}[small,axis on top,
             xmax=1,xmin=0.01,
             xlabel={$h$},xmode=log,x dir=reverse,
             ylabel near ticks,ylabel={{\textnumero} of outer iterations},ymin=0,ymax=200
]
\DCplot table {img/cube/old_dual_conv_it.dat};
\end{axis}
\end{tikzpicture}
&
\begin{tikzpicture}
\begin{axis}[small,axis on top,
             xmax=1,xmin=0.01,
             xlabel={$h$},xmode=log,x dir=reverse,
             ylabel near ticks,ylabel={{\textnumero} of outer iterations},ymin=0,ymax=200
]
\Cplot table {img/cube/old_P1DP0_conv_it.dat};
\end{axis}
\end{tikzpicture}
\\
\begin{tikzpicture}
\begin{axis}[small,axis on top,
             xmax=1,xmin=0.01,
             xlabel={$h$},xmode=log,x dir=reverse,
             ylabel near ticks,ylabel={\tiny Average {\textnumero} of GMRES iterations},ylabel style={align=center},ymin=0,ymax=300
]
\DCplot table {img/cube/old_dual_conv_at.dat};
\end{axis}
\end{tikzpicture}
&
\begin{tikzpicture}
\begin{axis}[small,axis on top,
             xmax=1,xmin=0.01,
             xlabel={$h$},xmode=log,x dir=reverse,
             ylabel near ticks,ylabel={\tiny Average {\textnumero} of GMRES iterations},ylabel style={align=center},ymin=0,ymax=300
]
\Cplot table {img/cube/old_P1DP0_conv_at.dat};
\end{axis}
\end{tikzpicture}
\end{tabular}
\caption{The error, number of outer iterations and averge number of inner GMRES iteration
for the problem \cref{eq:Laplace} with boundary conditions \cref{sig:testbc2} on the unit cube as $h$ is reduced.
Here we take
$u_0=\lambda_0=0$, $\betaparam\D=0.01$, $\textsc{tol}=0.05$, $\textsc{maxiter}=200$,
and $\param=0.5/h$.
On the left (\DCdesc), we take $(u_n,\lambda_n),(v_h,\mu_h)\in\Poly_h^1(\Gamma)\times\DualPoly_h^0(\Gamma)$;
on the right (\Cdesc), we take $(u_n,\lambda_n),(v_h,\mu_h)\in\Poly_h^1(\Gamma)\times\DPoly_h^0(\Gamma)$.
The dashed lines show order 1 convergence (left) and order 1.5 convergence (right).
}
\label{figsig:cube_dual_conv}
\end{figure}

Motivated by \cref{figsig:cube_dual_param} and the bounds in
\cref{main_contact_result},
we take $\param=0.5/h$, and look at the convergence as $h$ is decreased.
\cref{figsig:cube_dual_conv} shows how the error and iteration counts vary as $h$ is decreased when
$\productspace{V}_h=\Poly_h^1(\Gamma)\times\DualPoly_h^0(\Gamma)$ (left, \DCdesc)
and
$\productspace{V}_h=\Poly_h^1(\Gamma)\times\DPoly_h^0(\Gamma)$ (right, \Cdesc).

For $\productspace{V}_h=\Poly_h^1(\Gamma)\times\DualPoly_h^0(\Gamma)$,
we observe slightly higher than the order 1 convergence predicted by \cref{main_contact_result}.
In this case, the mass matrix preconditioner is effective, as the number of GMRES iterations required inside each outer iteration is
reasonably low, and only grows slowly as $h$ is decreased. We believe that the effectiveness of the preconditioner for this choice of
spaces is due to the spaces $\Poly_h^1(\Gamma)$ and $\DualPoly_h^0(\Gamma)$ forming an inf-sup stable pair \cite[Lemma 3.1]{SteinbachDuals}.

When $\productspace{V}_h=\Poly_h^1(\Gamma)\times\DPoly_h^0(\Gamma)$,
\cref{main_contact_result} tells us to expect order 1.5 convergence.
However, we observe a slightly lower order.
This appears to be due to the ill-conditioning of this system, and the mass matrix preconditioner being ineffective,
leading to an inaccurate solution when using GMRES.
In this case, the spaces $\Poly_h^1(\Gamma)$ and $\DPoly_h^0(\Gamma)$ do not form an inf-sup stable pair, and so the mass-matrix between them is not guaranteed to be invertible
leading to a less effective preconditioner.


In order to obtain order 1.5 convergence with a well-conditioned system, 
we could look for $(u_h, \lambda_h)\in\Poly_h^1(\Gamma)\times\DPoly_h^0(\Gamma)$ and test
with $(v_h,\mu_h)\in\DualPoly_h^1(\Gamma)\times\DualPoly_h^0(\Gamma)$, where $\DualPoly_h^1(\Gamma)$ is the space of piecewise linear
functions on the dual grid that forms an inf-sup stable pair with the space $\DPoly_h^0(\Gamma)$, as defined in \cite{BC}.
With this choice of spaces, we obtain the higher order convergence as in \cref{main_contact_result},
while having stable dual pairings and hence more effective mass matrix preconditioning.

For the problems discussed in \cite{BeBuScro18}, we have run numerical experiments using this space pairing and observe the full order $\tfrac32$ convergence in a low number of iterations.
A deeper investigation of this method using these dual spaces, and the adaption of the theory to this case, warrants
future work.

\section{Conclusions}
Based on our work in \cite{BeBuScro18}, we have analysed and demonstrated the effectiveness of Nitsche type coupling methods for
boundary element formulations of contact problems.

An open problem is preconditioning. While the iteration counts in the presented examples were already 
practically useful, for large and complex structures preconditioning is still essential. The hope is to use
the properties of the Calder\'on projector to build effective operator preconditioning techniques for the
presented Nitsche type frameworks.

Avenues of future research include looking at how this approach could be applied to problems 
in linear elasticity, and an extension of this method to problems involving friction.

\bibliographystyle{siamplain}
\bibliography{contact}
\end{document}

%% file: img/grid1.tex
\begin{tikzpicture}[line cap=round,line join=round,>=angle 60,x=1.5cm,y=1.5cm]
\clip (3,0.6) rectangle (5.8,2.7);
\def\Gxtox{1}
\def\Gxtoy{0}
\def\Gytox{0.72}
\def\Gytoy{0.6}
\def\trihei{0.866}

\newcommand{\Gmake}[2]{{(#1)*\Gxtox+(#2)*\Gytox},{(#1)*\Gxtoy+(#2)*\Gytoy}}
\newcommand{\Gmakeup}[2]{{(#1)*\Gxtox+(#2)*\Gytox},{0.5+(#1)*\Gxtoy+(#2)*\Gytoy}}

\fill[DCcolor,opacity=0.1] (2,0) rectangle (6,4);
\fill[DCcolor,opacity=0.5] (2,0) rectangle (6,4);

\foreach \i in {0,...,6}
    \draw[line width=.4pt,white] (\Gmake{-3}{\i*\trihei}) -- (\Gmake{10}{\i*\trihei});
\foreach \i in {-2,...,4}
    \draw[line width=.4pt,white] (\Gmake{\i}{0}) -- (\Gmake{\i+4}{8*\trihei});
\foreach \i in {3,...,11}
    \draw[line width=.4pt,white] (\Gmake{\i}{0}) -- (\Gmake{\i-4}{8*\trihei});

\end{tikzpicture}

%% file: img/grid2.tex
\begin{tikzpicture}[line cap=round,line join=round,>=angle 60,x=1.5cm,y=1.5cm]
\clip (3,0.6) rectangle (5.8,2.7);
\def\Gxtox{1}
\def\Gxtoy{0}
\def\Gytox{0.72}
\def\Gytoy{0.6}
\def\trihei{0.866}

\newcommand{\Gmake}[2]{{(#1)*\Gxtox+(#2)*\Gytox},{(#1)*\Gxtoy+(#2)*\Gytoy}}
\newcommand{\Gmakeup}[2]{{(#1)*\Gxtox+(#2)*\Gytox},{0.5+(#1)*\Gxtoy+(#2)*\Gytoy}}

\fill[DCcolor,opacity=0.1] (2,0) rectangle (6,4);
\fill[DCcolor,opacity=0.5] (2,0) rectangle (6,4);

\foreach \i in {0,...,6}
    \draw[line width=.4pt,white] (\Gmake{-3}{\i*\trihei}) -- (\Gmake{10}{\i*\trihei});
\foreach \i in {-2,...,4}
    \draw[line width=.4pt,white] (\Gmake{\i}{0}) -- (\Gmake{\i+4}{8*\trihei});
\foreach \i in {3,...,11}
    \draw[line width=.4pt,white] (\Gmake{\i}{0}) -- (\Gmake{\i-4}{8*\trihei});

\foreach \i in {0,...,9}
    \draw[line width=.4pt,white,dashed] (\Gmake{\i/2}{0*\trihei}) -- (\Gmake{\i/2}{6*\trihei});

\foreach \i in {-2,...,9}
    \draw[line width=.4pt,white,dashed] (\Gmake{6+\i/4}{(\i/2)*\trihei}) -- (\Gmake{\i/4-1.5}{(5+\i/2)*\trihei});

\foreach \i in {-3,...,9}
    \draw[line width=.4pt,white,dashed] (\Gmake{0.5-\i/4}{(1+\i/2)*\trihei}) -- (\Gmake{5-\i/4}{(4+\i/2)*\trihei});

\end{tikzpicture}

%% file: img/grid3.tex
\begin{tikzpicture}[line cap=round,line join=round,>=angle 60,x=1.5cm,y=1.5cm]
\clip (3,0.6) rectangle (5.8,2.7);
\def\Gxtox{1}
\def\Gxtoy{0}
\def\Gytox{0.72}
\def\Gytoy{0.6}
\def\trihei{0.866}

\newcommand{\Gmake}[2]{{(#1)*\Gxtox+(#2)*\Gytox},{(#1)*\Gxtoy+(#2)*\Gytoy}}
\newcommand{\Gmakeup}[2]{{(#1)*\Gxtox+(#2)*\Gytox},{0.5+(#1)*\Gxtoy+(#2)*\Gytoy}}

\fill[DCcolor,opacity=0.1] (2,0) rectangle (6,4);
\fill[DCcolor,opacity=0.5] (2,0) rectangle (6,4);

\foreach \i in {-3,...,3}
    \foreach \j in {-3,...,3}
        \draw[line width=.4pt,white] (\Gmake{3+\i-\j/2}{(2.666666+\j)*\trihei}) -- (\Gmake{3+\i-\j/2}{(3.333333+\j)*\trihei})
                                  -- (\Gmake{2.5+\i-\j/2}{(3.666666+\j)*\trihei}) -- (\Gmake{2+\i-\j/2}{(3.333333+\j)*\trihei});

\end{tikzpicture}